[1994/12/01]
\documentclass{ijmart-mod}
\chardef\bslash=`\\ 

\usepackage{bm}
\usepackage{graphicx}
\usepackage[breaklinks=true]{hyperref}
\usepackage{mathtools}
\usepackage{caption}
\usepackage{amsmath}
\usepackage{array}
\usepackage{multirow}
\usepackage{soul}

\newtheorem{thm}{Theorem}[section]

\newtheorem{lem}[thm]{Lemma}
\newtheorem{prop}[thm]{Proposition}

\theoremstyle{definition}

\newtheorem{rem}[thm]{Remark}

\theoremstyle{remark}

\newcommand{\eval}[2][\right]{\relax
  \ifx#1\right\relax \left.\fi#2#1\rvert}

\begin{document}
\title{Kissing polytopes in dimension $3$}

\author[A. Deza]{Antoine Deza}
\address{McMaster University, Hamilton, Ontario, Canada}
\email{deza@mcmaster.ca} 

\author[Z. Liu]{Zhongyuan Liu}
\address{McMaster University, Hamilton, Ontario, Canada}
\email{liu164@mcmaster.ca} 

\author[L. Pournin]{Lionel Pournin}
\address{Universit{\'e} Paris 13, Villetaneuse, France}
\email{lionel.pournin@univ-paris13.fr}

\begin{abstract}
It is shown that the smallest possible distance between two disjoint lattice polytopes contained in the cube $[0,k]^3$ is exactly
$$
\frac{1}{\sqrt{2(2k^2-4k+5)(2k^2-2k+1)}}
$$
for every integer $k$ at least $4$. The proof relies on modeling this as a minimization problem over a subset of the lattice points in the hypercube $[-k,k]^9$. A precise characterization of this subset allows to reduce the problem to computing the roots of a finite number of degree at most $4$ polynomials, which is done using symbolic computation.
\end{abstract}


\maketitle

\section{Introduction}\label{DLP2.sec.1}

A polytope is the convex hull of finitely many points from $\mathbb{R}^d$ and in the special case when these points form a subset of $\mathbb{Z}^d$, this object is referred to as a \emph{lattice polytope} or sometimes as an integral polytope. Lattice polytopes appear in areas as diverse as optimization \cite{BarbatoGrappeLacroixLancini2023,CardinalSteiner2025,DelPiaMichini2016,DelPiaMichini2022,DezaPournin2018}, combinatorics \cite{BaranyVershik1991,BeckPixton2003,CardinalPournin2025,DavidPourninRakotonarivo2020,DezaPournin2022,KnauerMartnez-SandovalRamirezAlfonsin2018}, or algebraic topology \cite{Fulton1993,Laplante-Anfossi2022,Loday2004,MasudaThomasTonksVallette2021}. They are integer multiples of rational polytopes and as such, they can serve as convex body approximations for computational procedures \cite{BeckRobins2015}. The stopping criterion for certain such procedures depends on how close two such disjoint lattice polytopes can be, as for example in von Neumann's alternating projections algorithm that decides whether two convex bodies are disjoint \cite{BraunPokuttaWeismantel2022,VonNeumann1949}. Without any additional constraint, the distance between two disjoint lattice polytopes $P$ and $Q$ can be arbitrarily small but such constraints arise in practice. One can ask for example for how close can two disjoint lattice polytopes be, provided that the combined size of their binary encoding is bounded by a constant \cite{BraunPokuttaWeismantel2022,DezaOnnPokuttaPournin2024,Schrijver1998}.

\begin{table}[b]
\begin{center}
\begin{tabular}{>{\centering}p{0.8cm}cccc}
\multirow{2}{*}{$d$}& \multicolumn{4}{c}{$k$}\\
\cline{2-5}
 & $1$ & $2$  & $3$  & $k\geq4$\\
\hline & \\[-1.1\bigskipamount]
$2$ & $\sqrt{2}$ & $\sqrt{5}$ & $\sqrt{13}$  & $\sqrt{(k-1)^2+k^2}$\\
$3$ & $\sqrt{6}$ & $5\sqrt{2}$ & $\sqrt{299}$ & $\bm{\sqrt{2(2k^2-4k+5)(2k^2-2k+1)}}$\\
$4$ & $3\sqrt{2}$ & $2\sqrt{113}$ & $11\sqrt{71}$\\
$5$ & $\sqrt{58}$\\
$6$ & $\sqrt{202}$\\
\end{tabular}
\end{center}
\caption{The known values of $1/\varepsilon(d,k)$. The formula shown in bold is provided by Theorem \ref{DLP2.sec.1.thm.1}.}\label{DLP2.sec.1.tab.1}
\end{table}

A similar, combinatorial constraint is to require that $P$ and $Q$ are contained in the hypercube $[0,k]^d$ where $k$ is a fixed positive integer. Throughout the article, we will refer to such polytopes as \emph{lattice $(d,k)$-polytopes}. Since, there is only finitely many pairs of disjoint lattice $(d,k)$-polytopes, the smallest possible distance $\varepsilon(d,k)$ is well defined and one can ask for its value. We call \emph{kissing polytopes} two lattice $(d,k)$-polytopes whose distance is exactly $\varepsilon(d,k)$: even though they do not touch, they cannot get any closer. Lower and upper bounds on $\varepsilon(2,k)$ that are almost matching as $d$ goes to infinity have been given in \cite{DezaOnnPokuttaPournin2024} and a formula for $\varepsilon(2,k)$ in \cite{DezaLiuPournin2024} along with the exact value of $\varepsilon(d,k)$ when $d$ and $k$ are sufficiently small for the computations to be tractable.

One may alternatively ask, also under relevant constraints, for how flat a lattice polytope can be or for how close one of its faces can be from its other vertices. These questions arise from continuous optimization or combinatorial problems and have been studied for instance in \cite{AlonVu1997,BeckShtern2017,GutmanPena2018,Lacoste-JulienJaggi2015,Pena2019}.

We extend the ideas and techniques from \cite{DezaLiuPournin2024} to the $3$-dimensional situation and establish the following formula for $\varepsilon(3,k)$.

\begin{thm}\label{DLP2.sec.1.thm.1}
If $k$ is not equal to $3$, then
\begin{equation}\label{DLP2.sec.1.thm.1.eq.1}
\varepsilon(3,k)=\frac{1}{\sqrt{2(2k^2-4k+5)(2k^2-2k+1)}}\mbox{.}
\end{equation}
\end{thm}

All the known values of $\varepsilon(d,k)$ are reported in Table \ref{DLP2.sec.1.tab.1} and one can see that~(\ref{DLP2.sec.1.thm.1.eq.1}) does not hold when $k$ is equal to $3$. In the case when $k$ is at least $6$ we will also show that, up to the symmetries of the cube, $\varepsilon(3,k)$ is uniquely achieved by the pair $P^\star$ and $Q^\star$ of line segments such that the vertices of $P^\star$ are the lattice points $(k,2,1)$ and $(0,k-1,k)$ while $Q^\star$ has for its extremities, the origin of $\mathbb{R}^3$ and the lattice point $(k-1,k,k)$. These two line segments are depicted in Figure \ref{DLP2.sec.1.fig.1} when $k$ is equal to $2$ and when $k$ is at least $4$. The figure also shows pairs of line segments that achieve $\varepsilon(3,1)$ and $\varepsilon(3,3)$.

It is shown in \cite{DezaOnnPokuttaPournin2024} that $\varepsilon(d,k)$ is always achieved as the distance between two lattice $(d,k)$-simplices whose dimensions sum to $d-1$. In the $2$-dimensional case, it therefore suffices to consider a point and a line segment. In the $3$\nobreakdash-dimensional case however, we need to consider both the distance between a point and a triangle and the distance between two line segments. As an intermediate step to proving Theorem \ref{DLP2.sec.1.thm.1}, we will show that the former case can be ignored.

\begin{figure}[b]
\begin{centering}
\includegraphics[scale=1]{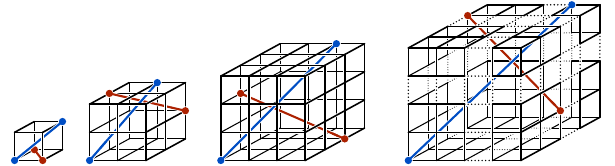}
\caption{Pairs of lattice $(3,k)$-segments that achieve $\varepsilon(3,k)$ for $k$ equal to $1$, $2$, $3$, and at least $4$ (from left to right).}\label{DLP2.sec.1.fig.1}
\end{centering}
\end{figure}

\begin{thm}\label{DLP2.sec.3.thm.0}
Consider a lattice point $P$ contained in $[0,k]^3$ and a lattice $(3,k)$-triangle $Q$. If $P$ is not contained in $Q$, then
$$
d(P,Q)>\varepsilon(3,k)\mbox{.}
$$
\end{thm}

As a preliminary to the proof of these theorems, we describe in Section \ref{DLP2.sec.2} the optimization model from \cite{DezaLiuPournin2024} that allows to provide lower bounds on $\varepsilon(d,k)$. We then use that model in Section \ref{DLP2.sec.3} to prove Theorem \ref{DLP2.sec.3.thm.0} and in Section \ref{DLP2.sec.4} to prove Theorem \ref{DLP2.sec.1.thm.1} by reducing the problem to studying a finite number of polynomial expressions using symbolic computation.

\section{A discrete optimization model}\label{DLP2.sec.2}

Consider two lattice $(d,k)$\nobreakdash-simplices $P$ and $Q$ whose dimensions sum to $d-1$. In~\cite{DezaLiuPournin2024}, the distance between $P$ and $Q$ is lower bounded as follows. Denote by $p^0$ to $p^n$ the vertices of $P$ and by $q^0$ to $q^m$ the vertices of $Q$. Consider the $d\mathord{\times}(d-1)$ matrix $A$ whose $j$th column is $p^j-p^0$ when $j$ is at most $n$ and $q^{j-n}-q^0$ otherwise. Further consider the vector $b$ equal to $q^0-p^0$. It is shown in \cite{DezaLiuPournin2024} (see Lemma 2 therein) that if $A^tA$ is non-singular, then
\begin{equation}\label{DLP2.sec.2.eq.1}
d\bigl(\mathrm{aff}(P),\mathrm{aff}(Q)\bigr)=\|A(A^tA)^{-1}A^tb-b\|
\end{equation}
where $\mathrm{aff}(P)$ and $\mathrm{aff}(Q)$ denote the affine hulls of $P$ and $Q$. It follows that the distance between $P$ and $Q$ is at least the right-hand side of (\ref{DLP2.sec.2.eq.1}).

\begin{rem}\label{DLP2.sec.2.rem.1}
Observe that $A$ and $b$ are by no means canonical since exchanging $P$ and $Q$ or relabeling their vertices will result in a different matrix $A$ and a different vector $b$. One can also consider $P$ and $Q$ up to the symmetries of the hypercube, which amounts to permuting or negating a subset of the rows of $A$ and performing the same transformation on the coordinates of $b$. Any such operation does not change (\ref{DLP2.sec.2.eq.1}) or the determinant of $A^tA$. As remarked in \cite{DezaLiuPournin2024} the same is true when a subset of the columns of $A$ is negated.
\end{rem}

This construction allows to provide a lower bound on $\varepsilon(d,k)$. Indeed, building on results from~\cite{DezaOnnPokuttaPournin2024} it is shown in~\cite{DezaLiuPournin2024} (see Lemma 2 and Proposition 4 therein) that $P$ and $Q$ can be chosen in such a way that the distance between these simplices is precisely $\varepsilon(d,k)$ while $A^tA$ is non-singular for any matrix $A$ built from $P$ and $Q$. It is then immediate that $\varepsilon(d,k)$ is at least the right-hand side of (\ref{DLP2.sec.2.eq.1}) when $A$ and $b$ correspond to such a pair of lattice simplices.

Let us now assume from now on that $d$ is equal to $3$. In that case, $A$ is a $3\mathord{\times}2$ matrix with integer coefficients and $b$ is a vector from $\mathbb{Z}^3$. Moreover, the absolute value of the coefficients of $A$ and $b$ is at most $k$ because they are differences of two non-negative numbers at most $k$. Now recall that $P$ and $Q$ are simplices whose dimensions sum to $2$. These two simplices are therefore either two line segments or a point and a triangle. In the latter case, we can assume without loss of generality that $P$ is the point while $Q$ is the triangle and by the above construction, the first column of $A$ is $q^1-q^0$ and its second column is $q^2-q^0$. In the former case, the first column of $A$ is $p^1-p^0$ and the second one is $q^1-q^0$. In both cases, we turn the matrix $A$ and the vector $b$ into the lattice point $x$ contained in the hypercube $[-k,k]^9$ whose coordinates $x_1$ to $x_6$ are obtained by identification from the coefficients of $A$ as
\begin{equation}\label{DLP2.sec.2.eq.2}
\left[
\begin{array}{cc}
x_1 & x_4\\
x_2 & x_5\\
x_3 & x_6\\
\end{array}
\right]=A
\end{equation}
and whose coordinates $x_7$ to $x_9$ from those of $b$ as
\begin{equation}\label{DLP2.sec.2.eq.2.1}
\left[
\begin{array}{cc}
x_7\\
x_8\\
x_9\\
\end{array}
\right]=b\mbox{.}
\end{equation}

In the sequel, we denote by $\mathcal{X}(k)$ the set of all the lattice points $x$ contained in $[-k,k]^9$ that can be obtained as we have just described from a pair of disjoint lattice $(3,k)$\nobreakdash-simplices $P$ and $Q$ whose dimensions sum to $2$.

Let us now consider an arbitrary $3\mathord{\times}2$ matrix $A$ and an arbitrary vector $b$ in $\mathbb{R}^3$, both with integer coefficients of absolute value at most $k$. Here, we no longer assume that $A$ and $b$ are obtained from a pair of lattice polytopes but we can still associate to them a lattice point $x$ from $[-k,k]^9$ via (\ref{DLP2.sec.2.eq.2}) and (\ref{DLP2.sec.2.eq.2.1}). In that case, the determinant of $A^tA$ is equal to $g(x)$ where
\begin{equation}\label{DLP2.sec.2.eq.0.5}
g(x)=(x_1x_5-x_2x_4)^2+(x_1x_6-x_3x_4)^2+(x_2x_6-x_3x_5)^2\mbox{.}
\end{equation}

Note that this expression for the determinant of $A^tA$ can be recovered from the Cauchy--Binet formula \cite[Example 10.31]{ShafarevichRemizov2013}. It is observed in \cite{DezaLiuPournin2024} that, when the matrix $A^tA$ is non-singular or equivalently when $g(x)$ is not equal to $0$, the right-hand side of~(\ref{DLP2.sec.2.eq.1}) can be expressed in terms of $x$ as
\begin{equation}\label{DLP2.sec.2.eq.3}
\|A(A^tA)^{-1}A^tb-b\|=\frac{|f(x)|}{\sqrt{g(x)}}
\end{equation}
where $f$ is the function of $x$ defined as 
\begin{equation}\label{DLP2.sec.2.eq.0}
f(x)=x_1(x_6x_8-x_5x_9)+x_2(x_4x_9-x_6x_7)+x_3(x_5x_7-x_4x_8)\mbox{.}
\end{equation}

The above mentioned results of \cite{DezaLiuPournin2024,DezaOnnPokuttaPournin2024} can therefore be collected into the following theorem in the special case when $d$ is equal to $3$.

\begin{thm}\label{DLP2.sec.2.thm.1}
For every positive integer $k$, there exists a lattice point $x$ in $\mathcal{X}(k)$ such that $f(x)$ is non-zero, $g(x)$ is positive, and
$$
\varepsilon(3,k)\geq\frac{|f(x)|}{\sqrt{g(x)}}\mbox{.}
$$
\end{thm}

Let us conclude the section with an upper bound on $\varepsilon(3,k)$ obtained from the two line segments $P^\star$ and $Q^\star$ that we have described in the introduction. These two line segments depend on $k$ and can be defined for any integer $k$ at least $2$. It is observed in \cite{DezaLiuPournin2024} that, for any such integer $k$, 
$$
d(P^\star,Q^\star)=\frac{1}{\sqrt{2(2k^2-4k+5)(2k^2-2k+1)}}
$$
and since $P^\star$ and $Q^\star$ are lattice $(3,k)$-polytopes,
\begin{equation}\label{DLP2.sec.3.eq.1}
\varepsilon(3,k)\leq\frac{1}{\sqrt{2(2k^2-4k+5)(2k^2-2k+1)}}\mbox{.}
\end{equation}

Note that when $k$ is at most $8$, the values of $\varepsilon(3,k)$ are reported in \cite{DezaLiuPournin2024}. In particular, $\varepsilon(3,1)$ is $1/\sqrt{6}$ and therefore, (\ref{DLP2.sec.3.eq.1}) still holds (with equality) when $k$ is equal to $1$. For this reason we shall use this inequality for all positive $k$ in the sequel. The values reported in \cite{DezaLiuPournin2024} further show that this bound is always sharp when $k$ is at most $8$ but different from $3$. We shall see that this upper bound is in fact also sharp for every integer $k$ greater than $8$.

\section{The case of a point and a triangle}\label{DLP2.sec.3}

The goal of this section is to prove that the distance between a lattice point $P$ contained in $[0,k]^3$ and a lattice $(3,k)$-triangle $Q$ that does not contain this point is greater than $\varepsilon(3,k)$ when $k$ is at least $8$. It follows in particular that, for any such value of $k$, it suffices to consider pairs of lattice $(3,k)$-segments in order to compute $\varepsilon(3,k)$. We will consider two separate cases depending on whether $P$ belongs to the affine hull of $Q$ or not. Indeed, when it does not, it follows from \cite[Lemma 2]{DezaLiuPournin2024} and \cite[Proposition 4]{DezaLiuPournin2024} that either the distance between $P$ and $Q$ is greater than $\varepsilon(3,k)$ or for any matrix $A$ and vector $b$ obtained from them as explained in Section \ref{DLP2.sec.2}, $A^tA$ is non-singular. In the latter case, we can use (\ref{DLP2.sec.2.eq.1}) and (\ref{DLP2.sec.2.eq.3}) in order to lower bound the distance of $P$ and $Q$, where $x$ is the lattice point in $[-k,k]^9$ obtained from $A$ and $b$ via (\ref{DLP2.sec.2.eq.2}) and (\ref{DLP2.sec.2.eq.2.1}).

The proof will make use of the following straightforward statement, established by computing the roots of $5k^4-24k^3+40k^2-28k+10$.

\begin{prop}\label{DLP2.sec.3.prop.1}
For every positive integer $k$,
$$
\frac{1}{\sqrt{3}k^2}>\frac{1}{\sqrt{2(2k^2-4k+5)(2k^2-2k+1)}}\mbox{.}
$$
\end{prop}
%

We first prove the following.

\begin{lem}\label{DLP2.sec.3.lem.1}
Consider a lattice point $P$ contained in the cube $[0,k]^3$ and a lattice $(3,k)$-triangle $Q$. If $P$ is not contained in the affine hull of $Q$, then the distance between $P$ and $Q$ is greater than $\varepsilon(3,k)$.
\end{lem}
\begin{proof}
Denote by $q^0$, $q^1$, and $q^2$ the vertices of $Q$. Up to the symmetries of the cube $[0,k]^3$, we can assume that all the coordinates of $q^2-q^0$ are non-negative. We will also consider the matrix $A$ and the vector $b$ obtained from $P$, $q^0$, $q^1$, and $q^2$ as explained in Section \ref{DLP2.sec.2}: the first column of $A$ is $q^1-q^0$ and its second column $q^2-q^0$ while $b$ is equal to $q^0-P$. Assume that $P$ does not belong to the affine hull of $Q$ and, for contradiction, that the distance between $P$ and $Q$ is $\varepsilon(3,k)$. In that case, according to \cite[Lemma 2]{DezaLiuPournin2024} and \cite[Proposition 4]{DezaLiuPournin2024} the matrix $A^tA$ is non-singular. Equivalently, $g(x)$ is positive where $x$ is the lattice point in the hypercube $[-k,k]^9$ whose coordinates are given by
$$
x_i=\left\{
\begin{array}{l}
q^1_i-q^0_i\mbox{ if }1\leq{i}\leq{3}\mbox{,}\\
q^2_{i-3}-q^0_{i-3}\mbox{ if }4\leq{i}\leq6\mbox{,}\\
q^0_{i-6}-P_{i-6}\mbox{ if }7\leq{i}\leq9\mbox{.}
\end{array}
\right.
$$

According to (\ref{DLP2.sec.2.eq.1}) and (\ref{DLP2.sec.2.eq.3}), the absolute value of $f(x)$ must be at least $1$ because $P$ does not belong to the affine hull of $Q$ and we obtain
$$
d(P,Q)\geq\frac{1}{\sqrt{g(x)}}\mbox{.}
$$

According to (\ref{DLP2.sec.2.eq.0.5}), $g(x)$ is a sum of three squares, each of the form
$$
(x_ix_{j+3}-x_jx_{i+3})^2
$$
where $i$ and $j$ are distinct and at most $3$. We will prove that any such square is at most $k^4$ and therefore that $g(x)$ is at most $3k^4$.

Since all the coordinates of $q^2-q^0$ are non-negative and at most $k$, so are $x_{i+3}$ and $x_{j+3}$. We consider different cases depending on the signs of $x_i$ and $x_j$. If $x_i$ and $x_j$ are both non-negative or both non-positive, then the products $x_ix_{j+3}$ and $x_jx_{i+3}$ cannot have opposite signs. Hence,
$$
(x_ix_{j+3}-x_jx_{i+3})^2\leq\max\{|x_ix_{j+3}|,|x_jx_{i+3}|\}^2
$$
and since $x$ belongs to the hypercube $[-k,k]^9$, it follows that the considered square is at most $k^4$. Now if $x_i$ is positive and $x_j$ is negative, then
\begin{equation}\label{DLP2.sec.3.lem.1.eq.1}
(x_ix_{j+3}-x_jx_{i+3})^2\leq(x_{j+3}-x_j)^2\max\{x_i,x_{i+3}\}^2\mbox{.}
\end{equation}

However, recall that $x_j$ is equal to $q^1_j-q^0_j$ and $x_{j+3}$ to $q^2_j-q^0_j$. As a consequence, $x_{j+3}-x_j$ is equal to $q^2_j-q^1_j$ and it follows that the absolute value of this difference is at most $k$. Hence, by (\ref{DLP2.sec.3.lem.1.eq.1}) the considered square is at most $k^4$ again. Finally if $x_i$ is negative and $x_j$ is positive then
$$
(x_ix_{j+3}-x_jx_{i+3})^2\leq(x_i-x_{i+3})^2\max\{x_j,x_{j+3}\}^2
$$
and the same argument (where $i$ and $j$ are exchanged) proves that the considered square is at most $k^4$ as well. We have therefore shown that
$$
d(P,Q)\geq\frac{1}{\sqrt{3}k^2}\mbox{.}
$$

By (\ref{DLP2.sec.3.eq.1}) and Proposition~\ref{DLP2.sec.3.prop.1}, this implies that the distance between $P$ and $Q$ is greater than $\varepsilon(3,k)$, which contradicts our assumption.
\end{proof}

There remains to treat the case when $P$ is contained in the affine hull of $Q$. This is a consequence of \cite[Theorem 5.1]{DezaOnnPokuttaPournin2024} that states that $\varepsilon(d,k)$ is a decreasing function of $d$ for every fixed $k$. In particular, $\varepsilon(2,k)$ is greater than $\varepsilon(3,k)$.

\begin{lem}\label{DLP2.sec.3.lem.2}
Consider a lattice point $P$ contained in the cube $[0,k]^3$ and a lattice $(3,k)$-triangle $Q$. If $P$ is contained in the affine hull of $Q$, then the distance between $P$ and $Q$ is greater than $\varepsilon(3,k)$.
\end{lem}
\begin{proof}
Consider a non-zero normal vector $a$ to the affine hull of $Q$. This vector has at least one non-zero coordinate and, up to permuting the coordinates of $\mathbb{R}^3$, it can be assumed without loss of generality that this coordinate is the third one. Let us identify $\mathbb{R}^2$ with the plane spanned by the first two coordinates of $\mathbb{R}^3$. Since the third coordinate of $a$ is non-zero, the orthogonal projection $\pi:\mathrm{aff}(Q)\rightarrow\mathbb{R}^2$ is a bijection. In addition $\pi$ sends a lattice $(3,k)$-polytope contained in the affine hull of $Q$ to a lattice $(2,k)$-polytope. Hence, $\pi(P)$ and $\pi(Q)$ are two disjoint lattice $(2,k)$-polytopes and
\begin{equation}\label{DLP2.sec.3.lem.2.eq.2}
d\bigl(\pi(P),\pi(Q)\bigr)\geq\varepsilon(2,k)\mbox{.}
\end{equation}

As $\pi$ is an orthogonal projection the distance of two points in the affine hull of $Q$ cannot be less than the distance of their images by $\pi$ and by (\ref{DLP2.sec.3.lem.2.eq.2}),
$$
d(P,Q)\geq\varepsilon(2,k)\mbox{.}
$$

According to \cite[Theorem 5.1]{DezaOnnPokuttaPournin2024}, $\varepsilon(2,k)$ is greater than $\varepsilon(3,k)$, which proves that the distance between $P$ and $Q$ is greater than $\varepsilon(3,k)$.
\end{proof}

Observe that Theorem \ref{DLP2.sec.3.thm.0} is an immediate consequence of Lemmas \ref{DLP2.sec.3.lem.1} and \ref{DLP2.sec.3.lem.2}. It follows in particular from Theorem \ref{DLP2.sec.3.thm.0} that $\varepsilon(3,k)$ is necessarily achieved as the distance between two lattice $(3,k)$-segments.

\section{The case of two line segments}\label{DLP2.sec.4}

We have shown in Section \ref{DLP2.sec.2} that $\varepsilon(3,k)$ can only be achieved as the distance between two lattice $(3,k)$-segments. This allows us to refine Theorem \ref{DLP2.sec.2.thm.1} into the following statement where $\mathcal{Y}(k)$ denotes the subset of the lattice points $x$ in the hypercube $[-k,k]^9$ such that $x_1$ is non-positive, $x_2$ to $x_6$ are non-negative, and for every integer $i$ satisfying $1\leq{i}\leq3$,
\begin{equation}\label{DLP2.sec.3.eq.2}
\left\{
\begin{array}{l}
|x_i-x_{i+6}|\leq{k}\mbox{,}\\
|x_{i+3}+x_{i+6}|\leq{k}\mbox{,}\\
|x_i-x_{i+3}-x_{i+6}|\leq{k}\mbox{.}\\
\end{array}
\right.
\end{equation}

As will be apparent from the proof of that statement, all the points $x$ in $\mathcal{X}(k)$ that correspond to a pair $P$ and $Q$ of lattice $(3,k)$-segments via the construction of Section \ref{DLP2.sec.2} satisfy (\ref{DLP2.sec.3.eq.2}). It can be proven that the converse is also true but we will not make use of that property in the sequel.

\begin{thm}\label{DLP2.sec.3.thm.1}
For every positive integer $k$, there exists a lattice point $x$ in $\mathcal{X}(k)\cap\mathcal{Y}(k)$ such that $f(x)$ is non-zero, $g(x)$ is positive, and
\begin{equation}\label{DLP2.sec.4.thm.1.eq.1}
\varepsilon(3,k)\geq\frac{|f(x)|}{\sqrt{g(x)}}\mbox{.}
\end{equation}
\end{thm}
\begin{proof}
According to \cite[Theorem 5.2]{DezaOnnPokuttaPournin2024}, there exists two lattice $(3,k)$-polytopes $P$ and $Q$ whose dimensions sum to $2$, whose affine hulls are disjoint, and whose distance is equal to $\varepsilon(3,k)$. It follows from Theorem \ref{DLP2.sec.3.thm.0} that both $P$ and $Q$ are line segments. Denote by $p^0$ and $p^1$ the vertices of $P$ and by $q^0$ and $q^1$ those of $Q$. Up to the symmetries of the cube $[0,k]^3$, we can assume that the coordinates of $q^1-q^0$ are all non-negative: if $q^1_i-q^0_i$ is negative, it suffices to replace $P$ and $Q$ by their symmetric with respect to the plane
$$
\biggl\{x\in\mathbb{R}^2:x_i=\frac{k}{2}\biggr\}\mbox{.}
$$

We can also assume without loss of generality that $p^1-p^0$ has at least two non-negative coordinates by exchanging $p^0$ and $p^1$ if needed and that
$$
p^1_1-p^0_1\leq{p^1_2-p^0_2}\leq{p^1_3-p^0_3}
$$
by permuting the coordinates of $\mathbb{R}^3$ appropriately. Note that any such permutation does not disturb the non-negativity of the coordinates of $q^1-q^0$.

Consider the matrix $A$ whose first column is $p^1-p^0$ and whose second column is $q^1-q^0$ and denote by $b$ the vector $q^0-p^0$. Further denote by $x$ the lattice point in $\mathcal{X}(k)$ obtained from $A$ and $b$ via (\ref{DLP2.sec.2.eq.2}) and (\ref{DLP2.sec.2.eq.2.1}). As the distance of $P$ and $Q$ is $\varepsilon(3,k)$ and their affine hulls are disjoint, it follows from \cite[Lemma 2]{DezaLiuPournin2024} and \cite[Proposition 4]{DezaLiuPournin2024} that $A^tA$ is non-singular. Since $g(x)$ is the determinant of $A^tA$, it must be non-zero. According to (\ref{DLP2.sec.2.eq.1}) and (\ref{DLP2.sec.2.eq.3}), $f(x)$ is non-zero as well because the affine hulls of $P$ and $Q$ are disjoint. Moreover,
\begin{equation}\label{DLP2.sec.3.thm.1.eq.1}
\varepsilon(3,k)\geq\frac{|f(x)|}{\sqrt{g(x)}}\mbox{.}
\end{equation}

Recall that the last two coordinates of $p^1-p^0$ and all the coordinates of $q^1-q^0$ are non-negative or, equivalently, that $x_2$ to $x_6$ are non-negative. We can show using an argument similar to that in the proof of Lemma \ref{DLP2.sec.3.lem.1} that $x_1$ must be negative. Indeed, assume for contradiction that $x_1$ is non-negative. In that case, the right-hand side of (\ref{DLP2.sec.2.eq.0.5}) is a sum of three squares, each of a difference of two non-negative terms. As a consequence,
$$
g(x)\leq\max\{x_1x_5,x_2x_4\}^2+\max\{x_1x_6,x_2x_4\}^2+\max\{x_2x_6,x_3x_5\}^2\mbox{.}
$$

However, as $x$ belongs to $[-k,k]^9$, it follows that $g(x)$ is at most $3k^4$ and since $f(x)$ is a non-zero integer, this and (\ref{DLP2.sec.3.thm.1.eq.1}) imply that
$\varepsilon(3,k)$ is at least $1/(\sqrt{3}k^2)$. Together with (\ref{DLP2.sec.3.eq.1}) and Proposition \ref{DLP2.sec.3.prop.1}, this results in a contradiction.

There remains to show that the point $x$ satisfies (\ref{DLP2.sec.3.eq.2}) for every integer $i$ such that $1\leq{i}\leq3$. Consider such an integer $i$ and recall that
$$
\left\{
\begin{array}{l}
x_i=p^1_i-p^0_i\mbox{,}\\
x_{i+3}=q^1_i-q^0_i\mbox{,}\\
x_{i+6}=q^0_i-p^0_i\mbox{.}\\
\end{array}
\right.
$$

As an immediate consequence,
$$
\left\{
\begin{array}{l}
x_i-x_{i+6}=p^1_i-q^0_i\mbox{,}\\
x_{i+3}+x_{i+6}=q^1_i-p^0_i\mbox{,}\\
x_i-x_{i+3}-x_{i+6}=p^1_i-q^1_i\mbox{,}\\
\end{array}
\right.
$$
and since $p^0$, $p^1$, $q^0$, and $q^1$ belong to $[0,k]^3$, this shows that $x$ satisfies (\ref{DLP2.sec.3.eq.2}).
\end{proof}

Now denote by $\mathcal{Z}(k)$ the set of the points $x$ in the hypercube $[-k,k]^9$ (but not necessarily in $\mathbb{Z}^9$) such that $x_1$ is non-positive, $x_2$ to $x_6$ are non-negative, and for every integer $i$ satisfying $1\leq{i}\leq3$, the absolute values of $x_i$ and $x_{i+3}$ cannot both be equal to $k$. We can prove the following property for all the points contained in $\mathcal{Y}(k)$ provided that $k$ is large enough.

\begin{thm}\label{DLP2.sec.4.thm.1}
Consider a lattice point $x$ in $\mathcal{Y}(k)$ such that $f(x)$ is non-zero, $g(x)$ is positive. If $k$ is at least $5$ and $x$ satisfies (\ref{DLP2.sec.4.thm.1.eq.1}), then the absolute value of $f(x)$ is equal to $1$ and $x$ belongs to $\mathcal{Z}(k)$.
\end{thm}
\begin{proof}
Assume that $k$ is at least $5$. We begin by showing that the absolute value of $f(x)$ is equal to $1$. Assume for contradiction that this is not the case. As $f(x)$ is a non-zero integer, its absolute value must then be at least $2$. However, the absolute values of the coordinates of $x$ are all at most~$k$ and it follows from~(\ref{DLP2.sec.2.eq.0.5}) that $g(x)$ is at most $12k^4$. Therefore,~(\ref{DLP2.sec.4.thm.1.eq.1}) implies
$$
\frac{|f(x)|}{\sqrt{g(x)}}\geq\frac{1}{\sqrt{3}k^2}
$$
which together with (\ref{DLP2.sec.4.thm.1.eq.1}) and Proposition \ref{DLP2.sec.3.prop.1} contradicts (\ref{DLP2.sec.3.eq.1}).

Now consider an integer $i$ such that $1\leq{i}\leq3$. We will show that $|x_i|$ and $|x_{i+3}|$ cannot both be equal to $k$. Assume, again for contradiction, that these two absolute values are equal to $k$ and let us first show that $x_{i+6}$ must be an integer multiple of $k$ as well. Since $x$ belongs to $\mathcal{Y}(k)$ its first coordinate is non-positive and the next five are non-negative. Hence, if $i$ is equal to $1$, then $x_i$ is equal to $-k$ and $x_{i+3}$ to $k$. Therefore, according to the third inequality from (\ref{DLP2.sec.3.eq.2}), $2k+x_{i+6}$ is at most $k$ and it follows that $x_{i+6}$ is necessarily equal to $-k$. If however, $i$ is equal to $2$ or $3$, then $x_i$ and $x_{i+3}$ are both equal to $k$. In that case, by the first inequality from~(\ref{DLP2.sec.3.eq.2}), $x_{i+6}$ cannot be negative and according to the second inequality, it cannot be positive, which shows that it must be equal to $0$. As a consequence, $x_i$, $x_{i+3}$, and $x_{i+6}$ all are integer multiples of $k$. However observe that~(\ref{DLP2.sec.2.eq.0}) can be rewritten the three following ways:
$$
\left\{
\begin{array}{l}
f(x)=x_1(x_6x_8-x_5x_9)+x_4(x_2x_9-x_3x_8)+x_7(x_3x_5-x_2x_6)\mbox{,}\\
f(x)=x_2(x_4x_9-x_6x_7)+x_5(x_3x_7-x_1x_9)+x_8(x_1x_6-x_3x_4)\mbox{,}\\
f(x)=x_3(x_5x_7-x_4x_8)+x_6(x_1x_8-x_2x_7)+x_9(x_2x_4-x_1x_5)\mbox{.}\\
\end{array}
\right.
$$

Since $x_i$, $x_{i+3}$, and $x_{i+6}$ are integer multiples of $k$, the $i$th equality among these three shows that $f(x)$ is also an integer multiple of $k$. Since $f(x)$ is not equal to zero and $k$ is not equal to $1$, this implies that the absolute value of $f(x)$ cannot be equal to $1$ and we reach a contradiction. As a consequence, the absolute values of $x_i$ and $x_{i+3}$ cannot both be equal to $k$.
\end{proof}

According to Theorem \ref{DLP2.sec.4.thm.1}, when $k$ is not too small it suffices to consider the points $x$ in $\mathcal{Y}(k)$ such that the absolute value of $f(x)$ is equal to $1$  in order to lower bound $\varepsilon(3,k)$, which given the right-hand side of (\ref{DLP2.sec.4.thm.1.eq.1}) amounts to maximize $g(x)$. Using this, we will further restrict the search space to a set of lattice points that does not depend on $k$. In particular, we will prove that the search for a lower bound on $\varepsilon(3,k)$ via Theorem \ref{DLP2.sec.3.thm.1} can be restricted to the points $x$ in $\mathcal{Y}(k)$ such that $h(x)$ is at least $6k-5$ where
$$
h(x)=-x_1+x_2+x_3+x_4+x_5+x_6\mbox{.}
$$

Consider the set $\mathcal{B}$ of the points $x$ in $\mathbb{N}^9$ whose first six coordinates sum to~$6$, whose last three coordinates are equal to zero, while $x_i+x_{i+3}$ is at least $1$ when $i$ satisfies $1\leq{i}\leq3$. Observe that $\mathcal{B}$ is a finite set. When $k$ at least $6$, we can embed $\mathcal{B}$ as a subset of $\mathcal{Z}(k)\cap\mathbb{Z}^9$ by using the affine map $\phi_k:\mathbb{R}^9\rightarrow\mathbb{R}^9$ such that the $i$th coordinate of $\phi_k(x)$ is given by
$$
\bigl[\phi_k(x)\bigr]_i=\left\{
\begin{array}{l}
-k+x_i\mbox{ if }i\mbox{ is equal to }1\mbox{ or }7\mbox{,}\\
k-x_i\mbox{ if }i\mbox{ is at least }2\mbox{ and at most }6\mbox{,}\\
x_i\mbox{ if }i\mbox{ is equal to }8\mbox{ or }9\mbox{.}
\end{array}
\right.
$$

Note in particular that $h(x)$ is equal to $6k-6$ for every point $x$ in $\phi_k(\mathcal{B})$.

\begin{thm}\label{DLP2.sec.4.thm.2}
Consider a lattice point $z$ in $\mathcal{Z}(k)$ such that $h(z)$ is at most $6k-6$. If $k$ is at least $6$, then there exists a point $x$ in $\mathcal{B}$ such that
$$
g\circ\phi_k(x)\geq{g(z)}\mbox{.}
$$
\end{thm}
\begin{proof}
Assume that $k$ is at least $6$. First observe that $g(z)$ and $h(z)$ do not depend on the last three coordinates of $z$ and we can therefore assume without loss of generality that the seventh coordinate of $z$ is $-k$ and that its two last coordinates are both equal to $k$. The proof is by induction on $h(z)$. By construction, $\phi_k$ sends $\mathcal{B}$ to the subset of the points in $\mathcal{Z}(k)\cap\mathbb{Z}^9$ whose image by $h$ is $6k-6$, whose seventh coordinate is $-k$ and whose last two coordinates are equal to $k$. Hence,  if $h(z)$ is equal to $6k-6$, then $z$ belongs to $\phi_k(\mathcal{B})$ and it suffices to take for $x$ the pre-image of $z$ by $\phi_k$.

Now assume that $h(z)$ is at most $6k-7$. Since $z$ is a lattice point in $\mathcal{Z}(k)$, $-z_1+z_4$ is at most $2k-1$. However, for every point $x$ in $\mathcal{Z}(k)$,
$$
\frac{\partial{g}}{\partial{x_1}}(x)=2x_1(x_5^2+x_6^2)-2x_4(x_2x_5+x_3x_6)
$$
is always non-positive and
$$
\frac{\partial{g}}{\partial{x_4}}(z)=2x_4(x_2^2+x_3^2)-2x_1(x_2x_5+x_3x_6)
$$
is always non-negative. If $-z_1+z_4$ is less than $2k-1$, decreasing by $1$ the first coordinate of $z$ or increasing by $1$ its fourth coordinate results in a point $z'$ in $\mathcal{Z}(k)$. By the sign of the above derivatives on the line segment with extremities $z$ and $z'$, the image of $z'$ by $g$ is at least $g(z)$. Moreover, $h(z')$ is greater than $h(z)$ by~$1$. As a consequence, the result follows by induction.

Now assume that $-z_1+z_4$ is equal to $2k-1$ and recall that $h(z)$ is less than $6k-6$. Therefore, $z_2$, $z_3$, $z_5$, and $z_6$ cannot all be at least $k-1$. Assume that $z_2$ is less than $k-1$. Increasing by $1$ the second coordinate of $z$ results in a lattice point $z'$ that is still contained in $\mathcal{Z}(k)$. However,
$$
\frac{\partial{g}}{\partial{x_2}}(x)=2x_2(x_4^2+x_6^2)-2x_5(x_1x_4+x_3x_6)
$$
for every point $x$ in $\mathcal{Z}(k)$. As $-z_1+z_4$ is equal to $2k-1$, every point $x$ that belongs to the line segment with extremities $z$ and $z'$ is such that either $x_1$ is equal to $1-k$ and $x_4$ to $k$ or $x_1$ is equal to $-k$ and $x_4$ to $k-1$. In particular, the product $x_1x_4$ is necessarily equal to $-k(k-1)$. However, the product $x_3x_6$ is at most $k(k-1)$ on that line segment and the above partial derivative is therefore non-negative. It follows that $g(z')$ is at least $g(z)$ and since $h(z')$ is greater than $h(z)$ by~$1$, the result follows by induction.

Given that, for every point $x$ in $\mathcal{Z}(k)$,
$$
\begin{array}{l}
\displaystyle\frac{\partial{g}}{\partial{x_3}}(x)=2x_3(x_4^2+x_5^2)-2x_6(x_1x_4+x_2x_5)\mbox{,}\\[\bigskipamount]
\displaystyle\frac{\partial{g}}{\partial{x_5}}(x)=2x_5(x_1^2+x_3^2)-2x_2(x_1x_4+x_3x_6)\mbox{, and}\\[\bigskipamount]
\displaystyle\frac{\partial{g}}{\partial{x_6}}(x)=2x_6(x_1^2+x_2^2)-2x_3(x_1x_4+x_2x_5)\mbox{,}
\end{array}
$$
the same argument by induction shows that the result also holds when $-z_1+z_4$ is equal to $2k-1$ and $z_3$, $z_5$, or $z_6$ is less than $k-1$.
\end{proof}

Observe that for each point $x$ in $\mathcal{B}$, $g\circ\phi_k(x)$ is a polynomial function of $k$ of degree at most $4$. Hence, according to Theorem \ref{DLP2.sec.4.thm.2}, maximizing $g(x)$ over the points in $\mathcal{Z}(k)$ whose image by $h$ is at most $6k-6$ amounts to compare the values in $k$ of a fixed number (that does not depend on $k$) of degree at most $4$ polynomials. Observe that the the square of the denominator of the right-hand side of (\ref{DLP2.sec.3.eq.1}) is the degree $4$ polynomial $8k^4-24k^3+40k^2-28k+10$. It turns out that this polynomial is always greater than $g\circ\phi_k(x)$ when $k$ is at least $6$. This can be checked using symbolic computation. Indeed, the roots of
\begin{equation}\label{DLP2.sec.4.eq.3}
8k^4-24k^3+40k^2-28k+10-g\circ\phi_k(x)
\end{equation}
can be explicitly determined for each point $x$ in $\mathcal{B}$ as well as its sign when $k$ is equal to $6$. The computations show in particular that the largest real root of (\ref{DLP2.sec.4.eq.3}) when $x$ ranges over $\mathcal{B}$ is less than $6$ and that this polynomial is always positive when $k$ is equal to $6$. This results in the following proposition. 

\begin{prop}\label{DLP2.sec.4.prop.1}
For every integer $k$ at least $6$ and every point $x$ in $\mathcal{B}$,
$$
g\circ\phi_k(x)<8k^4-24k^3+40k^2-28k+10\mbox{.}
$$
\end{prop}

Denote by $\mathcal{A}$ the set of the points $x$ in $\mathbb{N}^7\mathord{\times}\mathbb{Z}^2$ whose first six coordinates sum to at most $5$, whose last three coordinates satisfy
\begin{equation}\label{DLP2.sec.4.eq.2}
\left\{
\begin{array}{l}
x_7\leq{x_1+x_4}\mbox{,}\\
-x_2\leq{x_8}\leq{x_5}\mbox{,}\\
-x_3\leq{x_9}\leq{x_6}\mbox{,}\\
\end{array}
\right.
\end{equation}
and such that $x_i+x_{i+3}$ is at least $1$ when $1\leq{i}\leq3$. Again, $\mathcal{A}$ is a finite set. Moreover, $\phi_k(\mathcal{A})$ contains all the points $x$ in $\mathcal{Y}(k)\cap\mathcal{Z}(k)$ such that $h(x)$ is at least $6k-5$. Note in particular that for any point $x$ in $\mathbb{N}^7\mathord{\times}\mathbb{Z}^2$, if the coordinates of $\phi_k(x)$ satisfy~(\ref{DLP2.sec.3.eq.2}), then the coordinates of $x$ satisfy (\ref{DLP2.sec.4.eq.2}).

Combining Theorems \ref{DLP2.sec.3.thm.1}, \ref{DLP2.sec.4.thm.1}, and \ref{DLP2.sec.4.thm.2} with Proposition \ref{DLP2.sec.4.prop.1} makes it possible to provide, when $k$ is at least $6$, a lower bound on $\varepsilon(3,k)$ that only depends on $f\circ\phi_k(x)$ and $g\circ\phi_k(x)$ where $x$ ranges over $\mathcal{A}$.

\begin{thm}\label{DLP2.sec.4.thm.3}
For every integer $k$ at least $6$, there exists a point $x$ in $\mathcal{A}$ such that $f\circ\phi_k(x)$ is equal to $1$, $g\circ\phi_k(x)$ is positive, and
$$
\varepsilon(3,k)\geq\frac{1}{\sqrt{g\circ\phi_k(x)}}\mbox{.}
$$
\end{thm}
\begin{proof}
Assume that $k$ is at least $6$. By Theorems \ref{DLP2.sec.3.thm.1} and \ref{DLP2.sec.4.thm.1}, there exists a point $z$ in $\mathcal{Y}(k)\cap\mathcal{Z}(k)$ such that $|f(z)|$ is equal to $1$, $g(z)$ is positive, and 
\begin{equation}\label{DLP2.sec.4.thm.3.eq.1}
\varepsilon(3,k)\geq\frac{1}{\sqrt{g(z)}}\mbox{.}
\end{equation}

It follows from Theorem \ref{DLP2.sec.4.thm.2} and Proposition \ref{DLP2.sec.4.prop.1} that $h(z)$ is at least $6k-5$. Indeed, otherwise, these two results would imply that $g(z)$ is less than the square of the denominator in the right-hand side of (\ref{DLP2.sec.3.eq.1}). In that case, (\ref{DLP2.sec.4.thm.3.eq.1}) would contradict~(\ref{DLP2.sec.3.eq.1}). As a consequence, $z$ is a point in $\mathcal{Y}(k)\cap\mathcal{Z}(k)$ such that $h(z)$ is at least $6k-5$ and therefore, this point is contained in $\phi_k(\mathcal{A})$. Taking, for $x$ the preimage of $z$ by $\phi_k$ completes the proof.
\end{proof}

For every point $x$ in $\mathcal{A}$, both $f\mathord{\circ}\phi_k(x)$ and $g\mathord{\circ}\phi_k(x)$ are polynomial functions of $k$, the former being of degree at most $3$ and the latter of degree at most $4$. By Theorem \ref{DLP2.sec.4.thm.3}, one can obtain a lower bound on $\varepsilon(3,k)$ when $k$ is at least $6$ by computing these two polynomials for every point $x$ in $\mathcal{A}$ and by checking whether the first one is equal to $1$ or to $-1$ for certain values of $k$ and, among the points $x$ such that this property holds, to pick the one for which the value in $k$ of the second polynomial is maximal. This requires solving a finite number of polynomial equations of degree at most $4$. As a consequence, using symbolic computation, we obtain the following proposition.

\begin{prop}\label{DLP2.sec.4.prop.2}
All the points $x$ in $\mathcal{A}$ such that 
\begin{enumerate}
\item[(i)] $|f\circ\phi_k(x)|$ is equal to $1$ and
\item[(ii)] $g\circ\phi_k(x)$ is not less than $8k^4-24k^3+40k^2-28k+10$
\end{enumerate}
for some integer $k$ at least $6$ correspond, up to the transformations described in Remark \ref{DLP2.sec.2.rem.1} to the pair $P^\star$ and $Q^\star$ of segments described  in Section \ref{DLP2.sec.2}.
\end{prop}

There are precisely eight points $x$ in $\mathcal{A}$ that satisfy the assertions (i) and~(ii) in the statement of Proposition \ref{DLP2.sec.4.prop.2} for some integer $k$ at least $6$. These points are reported in Table \ref{DLP2.sec.4.tab.1} as vectors of coordinates. Since these points all correspond to the segments $P^\star$ and $Q^\star$ up to the transformations described in Remark~\ref{DLP2.sec.2.rem.1}, they must satisfy the assertions (i) and~(ii) in the statement of Proposition~\ref{DLP2.sec.4.prop.2} for every integer $k$ at least $6$ and not just for some of these integers. By this observation, Theorem \ref{DLP2.sec.1.thm.1} follows from Theorem \ref{DLP2.sec.4.thm.3} and Proposition \ref{DLP2.sec.4.prop.2}. In particular, according to the values of $\varepsilon(3,1)$, $\varepsilon(3,2)$, $\varepsilon(3,4)$ and $\varepsilon(3,5)$ reported in \cite{DezaLiuPournin2024}, the theorem indeed holds when $k$ is equal to $1$, $2$, $4$, or $5$ even though these four values are not covered by Theorem \ref{DLP2.sec.4.thm.3} and Proposition \ref{DLP2.sec.4.prop.2}.

\begin{table}[ht!]
\begin{tabular}{cc}
(0,1,3,1,0,0,0,-1,-2) & (0,1,3,1,0,0,1,0,-1)\\
(0,3,1,1,0,0,0,-2,-1) & (0,3,1,1,0,0,1,-1,0)\\
(1,0,0,0,1,3,0,1,2) & (1,0,0,0,1,3,1,0,1)\\
(1,0,0,0,3,1,0,2,1) & (1,0,0,0,3,1,1,1,0)\\
\end{tabular}
\caption{The eight lattice points $x$ in $\mathcal{A}$ such that for some integer $k$ at least $6$ both the assertion (i) and the assertion (ii) in the statement of Proposition \ref{DLP2.sec.4.prop.2} hold.}\label{DLP2.sec.4.tab.1}
\end{table}

\bibliography{KissingPolytopesDimension3}
\bibliographystyle{ijmart}

\end{document}